\documentclass[12pt,a4paper,reqno]{amsart}
\usepackage{amsfonts,amsthm,amsmath,amscd}
\usepackage[utf8]{inputenc}
\usepackage{t1enc}
\usepackage[mathscr]{eucal}
\usepackage{indentfirst}
\usepackage{tikz}
\usepackage{caption}
\usepackage{subcaption}
\usetikzlibrary{automata,arrows,positioning,calc}
\usepackage{graphicx,graphics,pict2e}
\usepackage{tkz-berge}
\usepackage{float}
\usepackage{epic}
\usepackage{lipsum}
\usepackage{stmaryrd}
\usepackage{authblk}
\usepackage{subcaption}
\usepackage[symbol]{footmisc}
\numberwithin{equation}{section}
\usepackage[margin=2.9cm]{geometry}
\usepackage{epstopdf}
\RequirePackage{doi}
\usepackage{hyperref}
\allowdisplaybreaks
\usepackage{cite}

\usepackage{verbatim,amssymb,amsfonts}
\usepackage{mathrsfs}
\usepackage{mathtools}
\usepackage{tikz-cd}
\usepackage{indentfirst}
\usepackage{slashed}
\usepackage{thmtools}
\usepackage{setspace}
\usepackage{enumerate}
\usepackage{accents}

\theoremstyle{plain}
\newtheorem{theorem}{Theorem}[section]

\newtheorem{corollary}[theorem]{Corollary}
\newtheorem{proposition}[theorem]{Proposition}

 \theoremstyle{definition}
\newtheorem{definition}[theorem]{Definition}
\newtheorem{remark}[theorem]{Remark}
\newtheorem{example}[theorem]{Example}

\DeclarePairedDelimiterX{\inp}[2]{\langle}{\rangle}{#1, #2}

\newcommand{\cE}{{\mathcal E}}

\newcommand{\cM}{{\mathcal M}}

\newcommand{\tr}{{\mathrm{Tr}}}

\newcommand{\ba}{\begin{eqnarray}}
\newcommand{\na}{\end{eqnarray}}
\newcommand{\ban}{\begin{eqnarray*}}
\newcommand{\nan}{\end{eqnarray*}}

\newcommand{\N}{{\mathbb N}}

\newcommand{\R}{{\mathbb R}}

\renewcommand{\thefootnote}{\fnsymbol{footnote}}

\makeatletter
\g@addto@macro{\endabstract}{\@setabstract}
\newcommand{\authorfootnotes}{\renewcommand\thefootnote{\@fnsymbol\c@footnote}}%
\makeatother

\makeatletter
\@namedef{subjclassname@2020}{%
  \textup{2020} Mathematics Subject Classification}
\makeatother

\title[]{closed walks of low dimension and twisted moments on self-loop graphs}



\begin{document}

\begin{center}
    \vspace{-1cm}
	\maketitle
	
	\normalsize
    \authorfootnotes
    Johnny Lim
	\par \bigskip



        \small{School of Mathematical Sciences, Universiti Sains Malaysia, Penang, Malaysia}\par \bigskip

\end{center}

\begin{abstract}
Let $G_S$ be a graph with loops attached at each vertex in $S \subseteq V(G).$ In this article, we develop exact formulae for the number of closed $3$- and $4$-walks on $G_S$ in terms of vertex degrees and certain elementary subgraphs of $G_S.$ We then derive the specific closed walks formulae for several graph families such as complete bipartite self-loop graphs, complete graphs, cycle graphs, etc. We demonstrate that such invariants are non-trivial in $G_S,$ which otherwise may be trivial in the loopless case. Moreover, we study a moment-like quantity
$\cM_q(G_S)=\sum^n_{i=1} |\lambda_i(G_S) - \frac{\sigma}{n}|^q,$ twisted by the spectral moment $\mathsf{M}_1(G_S)$ for $G_S,$ and show a positivity result. We also establish that the following ratio inequality holds:
\[
\frac{\cM_{1}}{\cM_{0}} \leq \frac{\cM_{2}}{\cM_{1}} \leq 
\frac{\cM_{3}}{\cM_{2}} \leq \frac{\cM_{4}}{\cM_{3}} \leq \cdots \leq \frac{\cM_{n}}{\cM_{n-1}} \leq \cdots.  
\] As a consequence, we obtain lower bounds for the self-loop graph energy $\cE(G_S)$ in terms of $\cM_i,$ extending some classical bounds.  
\end{abstract}

\section{Introduction}
\label{sect1}
Let $G$ be a simple graph of order $n=|V(G)|$ and size $m=|E(G)|,$ where $V(G)$ and $E(G)$ are the vertex and edge sets of $G,$ respectively. The degree of $v$ in $G$ is denoted by $d_G(v).$  
Let $K_n$ be the complete graph or order $n$; $K_{a,b}$ be the complete bipartite graph of part sizes $a$ and $b$; $P_n$ and $C_n$ be the path graph and cycle graph of order $n,$ respectively.  
By attaching one loop at each vertex in $S \subseteq V(G),$ we obtain a self-loop graph $G_S$ with $|S|=\sigma,$ and $0\leq \sigma \leq n.$ When $\sigma=0,$ $G_S=G;$ when $\sigma=n,$ we write $G_S=\widehat{G}.$
Recall that a multi-digraph is a pair $(\mathscr{V},\mathscr{E}),$ where $\mathscr{V}$ is a finite set of vertices and $\mathscr{E}$ is a set of ordered pairs of elements of $\mathscr{V}$, for which multiple edges and self-loops are allowed, cf. \cite{cvetkovic1995spectra}. Thus, we can regard a self-loop graph $G_S$ as a digraph without multiple edges, with any undirected edge being a pair of arcs connecting the same vertices but having opposite directions and conversely; we adopt the convention that any directed loop at a vertex corresponds to an undirected loop at the same vertex and vice versa. A walk $w_k$ of length $k$ in $G_S$ is a sequence of (not necessarily distinct) vertices $v_0,v_1,\ldots, v_k$ such that there is an edge from $v_{i-1}$ to $v_i$ for each $i=1,2,...,k.$ If $v_k=v_0,$ then the walk is said to be closed. We denote $w_k^{cl}(G_S)$ as the total number of closed walks of length $k$ in $G_S.$

Let $A(G_S)$ be the adjacency matrix of $G_S,$ i.e., $a_{ij}=1$ if $v_i$ is adjacent to $v_j$ ($i\neq j$), $a_{ii}=1$ for $v_i \in S,$ and $a_{ij}=0$ otherwise. The eigenvalues of $G_S$ are the eigenvalues of $A(G_S).$ Denote by $\lambda_1(G_S) \geq \lambda_2(G_S) \geq \cdots \geq \lambda_n(G_S)$ the eigenvalues of $G_S.$ In \cite{gutman2021energy}, the summation of these eigenvalues $\lambda_i(G_S)$ and their squares $\lambda_i^2(G_S)$ are obtained:
$\sum^n_{i=1} \lambda_i(G_S) = \sigma$ and $\sum^n_{i=1} \lambda_i^2(G_S) = 2m+\sigma.$
These formulae coincide\footnote[2]{We remark that in general cases (e.g. when the matrix is no longer the adjacency-type matrix) this interpretation of ``taking traces = counting closed walks'' is no longer true, see for instance \cite{SomborQuanti2023} for the discussion on the Sombor matrices of $(K_n)_S$.} with the number of closed 1- and 2-walks on $G_S,$ respectively, cf. Sect. \ref{sec2.1} below. 

On the other hand, closed $3$- and $4$-walks of $G_S$ remain an open topic. It was first realised in \cite{akbari2023selfloop} that closed 3-walks are necessary in determining the spectrum of complete bipartite graphs with self-loops. This motivates the first theme: \textit{to determine the number of closed 3-walks $w_3^{cl}$ and closed 4-walks $w_4^{cl}$ for any self-loop graph via combinatorial approach}. In Theorem \ref{closed3} and \ref{closed4GS}, we express $w_3^{cl}(G_S)$ and $w_4^{cl}(G_S),$ respectively, in terms of vertex degrees and some elementary subgraphs such as $K_2, K_3, K_4$ and $C_4,$ which provides a general computational method without resorting to specific adjacency matrices when $n$ is large. These two formulae extend the classical invariants $w_3^{cl}(G)$ (cf. \cite[Result 2h]{biggs1993algebraic}) and $w_4^{cl}(G)$ (cf. \cite{GutmanDas2004} and references therein) to self-loop graphs. We demonstrate its applications in Example 2.5-2.11 and Example 2.16-2.20.

The second part of this article is devoted to discussing a quantity generalized from spectral moments $\mathsf{M}_k= \sum^n_{i=1} \lambda^k_i$. Note that $\mathsf{M}_k(G)$ may be vanishing, e.g., when $G$ is a connected tree and $k$ is any odd positive integer. However, $\mathsf{M}_k(G_S)$ is \textit{a priori} at least $\sigma \geq 1$ for non-empty $S$ and $k\geq 1.$ 
Thus, we introduce a generalized moment-like quantity $\cM^k_q$, called \textit{twisted moments}, and investigate the case $k=1$ extensively in Sect. \ref{sect3}. This extends the results in \cite{Zhou2007} and simultaneously provides a way to obtain some bounds for the energy of $G_S,$ first introduced by Gutman \textit{et al.} \cite{gutman2021energy}: 
\begin{equation}
\label{eq:energy}
 \cE(G_S) = \sum^n_{i=1} \left| \lambda_i(G_S) - \frac{\sigma}{n}\right|.  
\end{equation}
The research on $G_S$ and its energy $\cE(G_S)$ is very current, see \cite{akbari2023selfloop,akbari2024line, jovanovic2023,rakshith2024,ShettyBhat2023} for new developments.  
The energy $\cE(G_S)$ can be viewed as a special case of a twisted moment when $k=q=1.$ 
In Sect. \ref{sect3}, we establish some basic results about $\cM_q (=\cM^1_q),$ including its determination for $q=0,1,2,3,4,$ and a positivity result for all $q.$ Then, we prove a ratio inequality of $\cM_q$ for $q \in \N \cup \{0\},$ which is used to derive bounds relating the quantities. 

\vspace{-0.2em}
\section{Closed $k$-walks on self-loop graphs for small values of $k$}
\label{sect2}
In the following subsections, we derive the formulae of closed $k$-walks for self-loop graphs for $k=2,3,$ and $4.$   Before that, recall the following theorem.
\begin{theorem}
\label{digraphwalk}
\cite{cvetkovic1995spectra}
Let $A$ be the adjacency matrix of a multi-digraph $G$ with vertices $1, 2, \ldots, n.$ Let $A^k=(a^{(k)}_{ij}).$ Moreover, let $w_k(i,j)$ denote the number of walks of length $k$ starting at the vertex $i$ and terminating at the vertex $j$.  Then, $w_{k}(i,j)=a^{(k)}_{ij}$ for $k=0,1,2,\ldots.$       
\end{theorem}
Here, the $0$-walk between $v_i$ and $v_j$ is $w_0(i,j)=\delta_{ij},$ where $\delta_{ij}=1$ if $i=j$ and $\delta_{ij}=0$ if $i \neq j.$ As explained in Sect. \ref{sect1}, we may consider a self-loop (undirected) graph as a digraph with self-loops and without multiple edges. Then, by Theorem~\ref{digraphwalk}, Schur's Triangularization Theorem \cite{HornJohnson2013}, and the conjugation invariance property of matrix trace, when $i=j,$  we have 
\[
\sum^n_{i=1} \lambda^k_i(G_S)
= \mathrm{Tr}(A^k(G_S)) 
= \sum^n_{i=1} a^{(k)}_{ii} =w^{cl}_k(G_S).
\]
\subsection{Closed 2-walks on $G_S$} 
\label{sec2.1}
It is clear that the number of closed $1$-walks for any self-loop graph $G_S$ is $\sigma,$ comprising only of a single self-looping, one for each loop.  Let us now study closed $2$-walks on $G_S$.

\begin{proposition}
Let $G_S$ be a connected self-loop graph of order $n\geq 2$ and size $m.$ Let $w_2^{cl}(G_S)$ be the total number of closed 2-walks on $G_S.$ Then,
\begin{equation}
\label{eq:closed2}
w_2^{cl}(G_S)= 2m + \sigma.
\end{equation}
\end{proposition}

\begin{proof}
Let $v_i,v_j \in V(G)$ such that $v_i\in S, v_j \in V(G) \backslash S.$ 
There are two possible closed $2$-walks at $v_i$: 
\begin{enumerate}[(i)]
\item Double self-looping $v_i \to v_i \to v_i,$ one walk for each loop.
\item Via an edge $v_i \to v_j \to v_i,$ two walks for each edge incident with $v_i.$  
\end{enumerate}
On the other hand, a closed 2-walk at $v_j$ only occurs via an edge. 
Thus, the total number of closed $2$-walks on $G_S$ is
\begin{align*}
w_2^{cl}(G_S)
&=\sum_{v_i \in S} (d_G(v_i)+1) + \sum_{v_i \notin S} d_G(v_i)  \\
&= \sum_{v\in V(G)} d_G(v) + \sum_{v_i\in S} 1 \\
&= 2m + \sigma. \qedhere
\end{align*}
\end{proof}

Equation~\ref{eq:closed2} coincides with  
$\sum^n_{i=1} \lambda^2_i(G_S)$ as derived in \cite[Lemma 4]{gutman2021energy}.
A similar method to the previous proof will be adopted to develop the number of closed 3-walks and 4-walks, respectively.
\subsection{Closed 3-walks on $G_S$} 

For $S \subseteq V(G),$ define
\begin{align}
n_1(v_i) &=\left|\{v_j \in V(G)\mid v_iv_j \in E(G) , v_i \in S, v_j \notin S\}\right|, \label{eq:n1a} \\
n_2(v_i) &=\left|\{v_j \in V(G)\mid v_iv_j \in E(G) , v_i,v_j \in S\}\right|, \label{eq:n2a} \\
n_\triangle(v_i) &= \left| \{\triangle(v_i,v_j,v_k) \mid v_i,v_j,v_k \in V(G)\} \right|, \label{eq:ntri}\\
n_{\triangle_1}(v_i) &= \left| \{\triangle(v_i,v_j,v_k) \mid v_i \in S, v_j,v_k \notin S\} \right| \label{eq:ntri1}\\
n_{\triangle_2}(v_i) &= \left| \{\triangle(v_i,v_j,v_k) \mid v_i, v_j \in S, v_k \notin S\} \right| \label{eq:ntri2} \\
n_{\triangle_3}(v_i) &= \left| \{\triangle(v_i,v_j,v_k) \mid v_i, v_j,  v_k \in S\} \right| \label{eq:ntri3} \\
n_{\square}(v_i) &= \left| \{\square(v_i,v_j,v_k,v_l) \mid v_i, v_j, v_k, v_l \in V(G) \wedge \notin V(K_4)\} \right| \label{eq:nC4}, \\
n_{\boxtimes}(v_i) &= \left| \{\boxtimes(v_i,v_j,v_k,v_l) \mid v_i, v_j, v_k, v_l \in V(G)\} \right| \label{eq:nK4}.
\end{align}
The quantity \eqref{eq:n1a} (resp. \eqref{eq:n2a}) corresponds to the number of edges incident with $v_i \in S$ and $v_j \notin S$ (resp. $v_j \in S$). Thus, 
\begin{equation}
\label{eq:n1n2}
n_1(v_i) + n_2(v_i) = d_G(v_i).
\end{equation} 
The quantity $n_\triangle(v_i)$ refers to the number of \textit{distinct} triangles $K_3$ \textit{at} $v_i,$ i.e., for which one of the vertices of the triangle is $v_i,$ whereas $n_{\triangle_r}(v_i), r=1,2,3,$ refer to the number of distinct triangles at $v_i$ such that each triangle has $r$ loops with $v_i$ having a loop. The last two quantities  $n_\square(v_i)$ and $n_\boxtimes(v_i)$ refer to the number of $C_4$ (not part of $K_4$) and $K_4$ at $v_i,$ respectively. For clarity, we remark that in a $K_4,$ we do not double count the ``boundary'' $C_4.$ In this case, if $v$ is any of its vertices, we write $n_{\boxtimes}(v)=1$ and $n_{\square}(v)=0.$

For notational brevity, in the following proof, we shall write  $ijkl$ to denote the walk $v_i \to v_j \to v_k \to v_l.$ 
\begin{theorem}
\label{closed3}
Let $G_S$ be a connected self-loop graph of order $n\geq 2$ and $|S|=\sigma.$ Let $w^{cl}_3(G_S)$ be the total number of closed 3-walks on $G_S.$ Then,
\begin{equation}
\label{eq:closed3}
w^{cl}_3(G_S)= 3 \sum_{v_i \in S} d_G(v_i) + 6 n_\triangle(G) + \sigma, 
\end{equation}
where $n_\triangle(G)$ is the number of triangles in $G.$ 
\end{theorem}

\begin{proof}
By definition, closed 3-walks on $G_S$ must traverse through either $(K_2)_S$ or $K_3$. 

\underline{Case 1}: Let $v_i \in S.$ There are four possibilities: 
\begin{enumerate}
    \item triple self-looping over $v_i:$ one walk ($iiii$),
    \item $(K_2)_S$ with vertices $v_i$ and $v_j \notin S:$ two walks each ($iiji$ and $ijii$), with a total of \textbf{$2n_1(v_i)$} walks;
    \item $(K_2)_S$ with vertices $v_i$ and $v_j \in S:$ three walks each ($iiji$, $ijii,$ and $ijji),$ with a total of \textbf{$3n_2(v_i)$} walks;
    \item $K_3$ with vertices $v_i,v_j,v_k:$ two walks each $(ijki, ikji)$ with a total of $2n_\triangle(v_i)$ walks. 
\end{enumerate}

\underline{Case 2}: Let $v_j \notin S.$ There are two possibilities:
\begin{enumerate}
    \item $(K_2)_S$ with vertices $v_j$ and $v_i \in S:$ one walk each ($jiij$) with a total of \textbf{$n_1(v_j)$} walks;
    \item $K_3$ with vertices $v_i,v_j,v_k:$ two walks each $(jikj, jkij)$ with a total of $2n_\triangle(v_j)$ walks. 
\end{enumerate}

Observe that each closed 3-walk on $(K_2)_S$ starting from $v_i \in S$ and with $v_j \notin S,$ corresponds to a closed 3-walk on $(K_2)_S$ starting from $v_j \notin S$ and with $v_i \in S,$ i.e.,
\begin{equation}
\label{eq:n1sum}
\sum_{v_i \in S} n_1(v_i) = \sum_{v_j \notin S} n_1(v_j).
\end{equation}
Therefore, the total number of closed 3-walks on $G_S$ is 
\begin{align*}
w_3^{cl}(G_S)
&= \sum_{v_i \in S} \left(2 n_1(v_i) + 3n_2(v_i) + 2n_\triangle(v_i) +1 \right) + 
\sum_{v_i \notin S} \left( n_1(v_i) + 2n_\triangle(v_i) \right) \\
&= \sum_{v_i \in S} \left(2 n_1(v_i) + 3n_2(v_i)\right) + \sum_{v_i \notin S} n_1(v_i) + \sum_{v \in V(G)} 2n_\triangle(v) + \sigma \\
&=\sum_{v_i \in S} \left(3 d_G(v_i) - n_1(v_i) \right) + \sum_{v_i \notin S} n_1(v_i) + 6 n_\triangle(G) + \sigma \\
&= 3\sum_{v_i \in S} d_G(v_i) + 6 n_\triangle(G) + \sigma. 
\end{align*}
where the third and fourth equalities follow from \eqref{eq:n1n2} and \eqref{eq:n1sum} respectively.
\end{proof}

\begin{remark}
\begin{enumerate}[(i)]
\item It is immediately to see that when $S=\emptyset$ (i.e., $\sigma=0$), Theorem~\ref{closed3} recovers the classical result $w^{cl}_3(G)=6n_\triangle(G).$
\item 
The third spectral moment of $G_S$ is thus given by
\[
\mathsf{M}_3(G_S)=\sum^n_{i=1} \lambda^3_i(G_S) = 3\sum_{v_i \in S} d_G(v_i) + 6 n_\triangle(G) + \sigma.
\]
\end{enumerate}
\end{remark}

\begin{example}
Let $G=K_n,$ $n \geq 3.$ Since
$
n_\triangle(K_n)= \begin{pmatrix}n \\ 3 \end{pmatrix} =  \dfrac{n!}{3!(n-3)!},
$
for any $S \subseteq V(G)$ with $|S|=\sigma,$  we obtain
\begin{align*}
w^{cl}_3((K_n)_S) 
&= 3 \sigma(n-1) + 6 \left(\frac{n!}{3!(n-3)!} \right) + \sigma\\
&=  \sigma(3n-2) + n(n-1)(n-2).
\end{align*}
Note that $w_3^{cl}((K_n)_S)$ is independent of the location of loops.
\end{example}

\begin{example}
Let $G=K_{a,b}$ be the complete graph of parts $(A,B)$ with size $a=|A|,b=|B| \geq 1.$ For $S = S_A \cup S_B \subseteq V(G)$ with $|S|=\sigma=\sigma_A + \sigma_B,$  since $n_{\triangle}(K_{a,b})=0,$ we deduce that 
\begin{align}
w^{cl}_3((K_{a,b})_S) 
&= 3 \left(\sum_{v_i \in S_A} d_G(v_i) + \sum_{v_i \in S_B} d_G(v_i) \right) + \sigma \nonumber\\
&= 3(b\sigma_A + a\sigma_B)+ \sigma. \label{eq:kmnS}
\end{align}
This is exactly the formula 
derived in \cite[Lemma 2.3]{akbari2023selfloop}, which has been applied to find the eigenvalues of complete bipartite self-loop graphs $(K_{a,b})_S$ when $0<\sigma<a$ and $a<\sigma<a+b$ \cite[Theorem 2.4]{akbari2023selfloop}, where the eigenvalues are exactly the root of some cubic polynomial determined by $w_3^{cl}.$ 
\end{example}

\begin{example}
\label{expPetersen}
Let $G=K(2k+1,k)$ be the Kneser graphs for  $k \geq 2.$  Note that $G$ is $\begin{pmatrix} 2k+1-k \\ k \end{pmatrix} = (k+1)$-regular. Since $2k+1 < 3k$ for $k \geq 2,$ we have $n_\triangle(G)=0.$ Then, for any $S \subseteq V(G),$ we deduce that 
\[
w^{cl}_3(K(2k+1,k)_S)  = 3\sigma(k+1) + \sigma = \sigma(3k+4).
\]
Let $PG$ be the Petersen graph, which is isomorphic to $K(5,2).$ Then, 
\[
w^{cl}_3((PG)_S)=10\sigma.
\] For example, consider only one loop at any vertex of $PG$, and without loss of generality we denote $1$ as the looped vertex and $2,3,4$ its adjacent vertices, then the ten closed 3-walks are 1111, 1122, 1211, 1131, 1311, 1141, 1411, 3113, 2112, and 4114. 
\end{example}

Example~\ref{expPetersen} illustrates that $w_3^{cl}(G_S)$ is a non-trivial invariant that depends on $\sigma \geq 1,$ which would otherwise be zero when $\sigma=0.$   Another similar observation is that if $G$ is a connected triangle-free graph, then $w_{3}^{cl}(G_S)$ is also non-zero for $\sigma \geq 1,$ see the next example.    

\begin{example}
Let $G$ be a graph of order $n.$ Suppose that $G$ has no triangles, then by \cite[Theorem 2.3]{Harary1969}, $G$ has at most $\frac{1}{4} n^2$ edges. Consider $G_S$ with $S\subseteq V(G), $ 
then we have
\[
0 \leq w^{cl}_3(G_S) \leq \frac{3}{2} n^2 + n,
\]
where the left equality holds when $\sigma=0,$ and the right equality holds when $G_S=\widehat{K_{\frac{n}{2},\frac{n}{2}}}:$ from \eqref{eq:kmnS} we have
\[
w_3^{cl}(\widehat{K_{{\frac{n}{2},\frac{n}{2}}}})= 3\left[ \left(\frac{n}{2}\right)^2 + \left(\frac{n}{2}\right)^2 \right] + n = \frac{3}{2} n^2 +n,
\]
where $n$ is even.
This aligns with Mantel's theorem, cf. \cite[\S 20, Theorem 3]{Aigner2018}.
\end{example}

\begin{example}
Let $G_S=(C_n)_S$ be the cycle graph of order $n \geq 3$ with $\sigma$ loops. Since $3\sum_{v \in S} d_G(v) = 6\sigma,$ $n_\triangle=1$ if $n=3$ and $n_\triangle=0$ if $n\geq 4.$  we get 
\[
w_3^{cl}((C_n)_S)
\begin{cases}
7\sigma + 6, &n=3, \\
7\sigma, &n\geq 4.
\end{cases}
\]
\end{example}

The following is an example of a connected graph with girth 3. 

\begin{example}
Let $W_n$ be a wheel graph of order $n$ with the ``center'' vertex $w_0.$  Recall that $W_n$ has $m=2(n-1)$ edges and $n-1$ triangles. Then,
\begin{align*}
w_3^{cl}((W_n)_S)
&= 
\begin{cases}
3(3(\sigma-1)+n-1) +6(n-1)+\sigma, & \text{for }w_0,w_1,\cdots, w_{\sigma-1} \in S, \\
3(3\sigma) + 6(n-1) + \sigma, & \text{for }w_1, \cdots, w_\sigma \in S, w_0 \notin S,
\end{cases}\\
&= 
\begin{cases}
10\sigma + 9(n -2), & \text{for }w_0,w_1,\cdots, w_{\sigma-1} \in S, \\
10\sigma + 6(n-1), & \text{for }w_1, \cdots, w_\sigma \in S, w_0 \notin S. 
\end{cases}
\end{align*}
    
\end{example}

\begin{example}
Let $G$ be a graph of order $2n$ and size $m=n^2+1.$ Such $G$ contains $n$ triangles (cf. \cite[pp 19]{Harary1969}). Consider $G_S$ with $S\subseteq V(G), |S|=\sigma,$ since $\sum_{v_i \in S}d_G(v_i) \leq 2m,$ we have
\[
 w^{cl}_3(G_S) \leq 6(n^2 +n+1) + \sigma.
\]
\end{example}

\subsection{Closed 4-walks on $G_S$}

Now, we discuss the number $w_4^{cl}(G_S)$ of closed $4$-walks on $G_S,$ which involves many more cases than that of $w_3^{cl}(G_S).$

\begin{theorem}
\label{closed4GS}
Let $G_S$ be a connected self-loop graph of order $n\geq 2,$ size $m \geq 1,$ and $|S|=\sigma.$ Let $w^{cl}_4(G_S)$ be the total number of closed 4-walks on $G_S.$ Then,
\begin{align}
w_4^{cl}(G_S) 
&=\sigma + 2(M_1(G) - m)
+ 6 \sum_{v_i \in S} d_G(v_i) - 2 \sum_{v_i \in S} n_1(v_i)  \nonumber \\
&\quad+ 8 (n_{\triangle_1}(G_S) + 2 n_{\triangle_2}(G_S) + 3 n_{\triangle_3}(G_S) + n_\square(G) + 3n_\boxtimes(G)), \label{eq:closed4GS}
\end{align}
where 
\begin{itemize}
\item $M_1(G)=\sum_{v \in V(G)} d_G^2(v)$ is the first Zagreb index of $G,$
\item $n_1(v_i)$ is the quantity \eqref{eq:n1a}, 
\item for $r=1,2,3$, $n_{\triangle_r}(G_S)$  is the number of triangles in $G_S$ such that $r$ vertex in $S$ and $3-r$ vertices in $V(G)\backslash S,$  
\item $n_\square(G)$ is the number of distinct $C_4$ in $G,$
\item $n_\boxtimes(G)$ is the number of distinct $K_4$ in $G.$
\end{itemize}
\end{theorem}

\begin{proof}

Let $v_i \in S.$ Then, there are eleven possible closed 4-walks at $v_i:$


\underline{Case 1}: 
One quadruple looping $(iiiii)$ at each $v_i,$ gives a total of $\sigma$ closed 4-walks. 
\vspace{0.5em}

\underline{Case 2}: 
For $v_j \in N(v_i),$ there is one closed 4-walk $(ijiji)$ at $v_i$ that traverses through an edge only and not via a loop, sums up to $d_G(v_i)$ walks. Overall, this case yields
$
\sum_{v \in V(G)} d_G(v)
$
closed 4-walks.
\vspace{1em}

\underline{Case 3}:
For $v_j,v_k \in N(v_i)$ and $v_j \neq v_k,$ there are two closed 4-walks $(ijiki, ikiji)$ at $v_i$ that traverse through two edges only and not via a loop, sums up to 
\[
2\begin{pmatrix} d_G(v_i) \\ d_G(v_i)-2 \end{pmatrix} = \frac{d_G(v_i)!}{(d_G(v_i)-2)!} = d_G(v_i) (d_G(v_i)-1)
\]
walks. Overall, this case yields 
$
\sum_{v \in V(G)} d_G(v)(d_G(v)-1)
$
closed 4-walks.
\vspace{1em}

\underline{Case 4}:
Consider a path $P_3$ with vertices $v_i,v_j,v_k$ such that $v_j \in N(v_i) \cap N(v_k)$. Then, there is one closed 4-walk $(ijkji)$ at $v_i$ that traverses via $P_3$ only and not via a loop nor a triangle. Thus, at $v_i$ we have 
$
\sum_{v_j \in N(v_i)} (d_G(v_j)-1) 
$
closed 4-walks.  
Summing over all $v_i \in V(G),$ we obtain
\[
\sum_{v_i \in V(G)}\left( \sum_{v_j \in N(v_i)}  d_G(v_j) \right) - \sum_{v_i \in V(G)}d_G(v_i) .
\]
\vspace{0.2em}

\underline{Case 5}:
Consider a $(K_2)_S$ with vertices $v_i \in S, v_j\notin S$. Then, there are three closed 4-walks ($iiiji$, $iijii$, $ijiii$) at $v_i$ and one closed 4-walk ($jiiij$) at $v_j \notin S,$ that must traverse through a loop and only one edge. The total number of closed 4-walks is 
\[
3 \sum_{v_i \in S} n_1(v_i) + \sum_{v_j \notin S} n_1(v_j).  
\] 
\vspace{0.2em}

\underline{Case 6}:
Consider a $(K_2)_S$ with vertices $v_i, v_j \in S$. Then, there are six closed 4-walks ($iiiji$, $iijii$, $ijiii$, $iijji$, $ijjii$, $ijjji$) at $v_i$ that traverse through at least one loop and an edge. The total number of closed 4-walks is 
\[
\sum_{v_i \in S} 6n_2(v_i).
\]

\underline{Case 7}:
Consider a $(K_3)_S$ with vertices $v_i \in S$ and $v_j,v_k \notin S.$ Then, there are four closed 4-walks ($iijki$, $iikji$, $ijkii$, $ikjii$) that traverse through a triangle and a loop.
Thus, the total closed 4-walks in this case is 
\[
\sum_{v_i \in S} 4 n_{\triangle_1}(v_i) + \sum_{v_i \notin S} 2n_{\triangle_1}(v_i) = 8n_{\triangle_2}(G_S).
\] 
\vspace{0.2em}

\underline{Case 8}:
Consider a $(K_3)_S$ with vertices $v_i, v_j \in S$ and $v_k \notin S.$ Then, there are six closed 4-walks ($iijki$, $iikji$, $ijkii$, $ikjii$, $ijjki$, $ikjji$) at $v_i$ that traverse through a triangle and a loop.  The total number of closed 4-walks is 
\[
\sum_{v_i \in S} 6 n_{\triangle_2}(v_i) + \sum_{v_i \notin S} 4n_{\triangle_2}(v_i) = 16n_{\triangle_2}(G_S).
\] 
\vspace{0.2em}

\underline{Case 9}:
Consider a $(K_3)_S$ with vertices $v_i, v_j, v_k \in S.$ Then, there are eight closed 4-walks ($iijki$, $iikji$, $ijkii$, $ikjii$, $ijjki$, $ikjji,$ $ikkji$, $ijkki$) at $v_i$ that traverse through a triangle and a loop. The total number of closed 4-walks is 
\[
\sum_{v_i \in S} 8 n_{\triangle_3}(v_i) = 24 n_{\triangle_3}(G_S).
\] 


\begin{figure}[h]
\vspace{-1em}
\begin{minipage}[b]{0.23\linewidth}
    \centering
      \begin{tikzpicture}[node distance={13mm}, very thin, main/.style = {draw, circle, fill, inner sep=1.5pt, minimum size=0pt}]  
                \node[main] (1) [label=left:$i$]{}; 
                \node[main] (2) [below left of=1]
                [label=left:$j$]{};
                \node[main] (3) [below right of=1][label=right:$k$]{};

            \draw (1) -- (2); 
            \draw (1) -- (3);
            \draw (2) -- (3);
        \end{tikzpicture}
    \caption*{$(a)$}
    \end{minipage}
\begin{minipage}[b]{0.23\linewidth}
    \centering
      \begin{tikzpicture}[node distance={13mm}, very thin, main/.style = {draw, circle, fill, inner sep=1.5pt, minimum size=0pt}] 
                \node[main] (1) [label=left:$i$]{}; 
                \node[main] (2) [below left of=1]
                [label=left:$j$]{};
                \node[main] (3) [below right of=1][label=right:$k$]{};

            \draw (1) -- (2); 
            \draw (1) -- (3);
            \draw (2) -- (3); 
            \draw (1) to [out=130,in=50,looseness=22] (1);
        \end{tikzpicture}
    \caption*{$(b)$}
\end{minipage}
\begin{minipage}[b]{0.23\linewidth}
    \centering
      \begin{tikzpicture}[node distance={13mm}, very thin, main/.style = {draw, circle, fill, inner sep=1.5pt, minimum size=0pt}] 
                \node[main] (1) [label=left:$i$]{}; 
                \node[main] (2) [below left of=1]
                [label=left:$j$]{};
                \node[main] (3) [below right of=1][label=right:$k$]{};

            \draw (1) -- (2); 
            \draw (1) -- (3);
            \draw (2) -- (3); 
            \draw (1) to [out=130,in=50,looseness=22] (1);
            \draw (3) to [out=120,in=40,looseness=20] (3);
        \end{tikzpicture}
    \caption*{$(c)$}
\end{minipage}
\begin{minipage}[b]{0.23\linewidth}
    \centering
      \begin{tikzpicture}[node distance={13mm}, very thin, main/.style = {draw, circle, fill, inner sep=1.5pt, minimum size=0pt}] 
                \node[main] (1) [label=left:$i$]{}; 
                \node[main] (2) [below left of=1]
                [label=left:$j$]{};
                \node[main] (3) [below right of=1][label=right:$k$]{};

            \draw (1) -- (2); 
            \draw (1) -- (3);
            \draw (2) -- (3); 
            \draw (1) to [out=130,in=50,looseness=22] (1);            
            \draw (2) to [out=140,in=60,looseness=20] (2);
            \draw (3) to [out=120,in=40,looseness=20] (3);
        \end{tikzpicture}
    \caption*{$(d)$}
\end{minipage}
\caption{$(K_3)_S$ with $|S|=0,1,2,$ and 3 respectively.}
\end{figure}

\underline{Case 10}:
Consider a $C_4$ with vertices $v_i,v_j,v_k,v_l$ (whether it has loops or not, labeled in a clockwise fashion accordingly). Then, there are two closed 4-walks ($ijkli$, $ilkji$) at $v_i$ that must traverse through all four vertices of $C_4.$ Thus, each vertex gives $2n_\square(v_i)$ walks that sums up to a total
\begin{equation}
\label{eq:obs4}
    \sum_{v_i\in S} 2n_{\square}(v_i) + \sum_{v_i \notin S} 2n_{\square}(v_i) = 8 n_\square(G).
\end{equation}
\vspace{0.2em}

\underline{Case 11}:
Consider a $K_4$ with vertices $v_i,v_j,v_k,v_l$ (whether has loops or not, labeled in a clockwise fashion accordingly). Then, there are six closed 4-walks ($ijkli$, $ilkji$, $ijlki$, $iljki$, $iklji$, $ikjli$) at $v_i$ that must traverse through all four vertices of $C_4.$ For each vertex we have $6 n_{\boxtimes}(v_i)$ walks, sums up to a total
    \begin{equation}
    \label{eq:obs5}
        \sum_{v_i\in S} 6n_{\boxtimes}(v_i) +  \sum_{v_i \notin S} 6n_{\boxtimes}(v_i) = 24 n_\boxtimes(G).
    \end{equation}


Some simplification can be done as follows:
\begin{itemize}
    \item Combining Case 2 and Case 4, we obtain
    \begin{equation}
    \label{eq:obs1}
    \sum_{v_i \in V(G)}\sum_{v_j \in N(v_i)}  d_G(v_j) = \sum_{v \in V(G)} d^2_G(v).
    \end{equation}
    
    \item Combining Case 5 and 6, we obtain 
    \begin{align}
    3 \sum_{v_i \in S} &n_1(v_i) + 6 \sum_{v_i \in S} n_2(v_i) + \sum_{v_i \notin S} n_1(v_i) \nonumber\\
    &= 6 \sum_{v_i \in S} d_G(v_i) - 3 \sum_{v_i \in S} n_1(v_i )+ \sum_{v_i \notin S} n_1(v_i) \nonumber\\
    &= 6 \sum_{v_i \in S} d_G(v_i) - 2 \sum_{v_i \in S} n_1(v_i ), \label{eq:obs2}
    \end{align}
    where we apply \eqref{eq:n1n2} in the first equality and \eqref{eq:n1sum} in the last equality.  
\end{itemize}

Finally, summing all possible cases above, we obtain
\begin{align*}
w_4^{cl}(G_S) 
&= \sigma + \left(\sum_{v \in V(G)} d_G(v) + \sum_{v_i \in V(G)} \sum_{v_j \in N(v_i)} (d_G(v_j)-1)\right) + \sum_{v \in V(G)} d_G(v)(d_G(v)-1)  \nonumber\\
&\quad+ \left(3\sum_{v_i \in S} n_1(v_i)  + 6 \sum_{v_i \in S} n_2(v_1) + \sum_{v_i \in S} n_1(v_i) \right) + 8n_{\triangle_1}(G_S) + 16n_{\triangle_2}(G_S) \\
&\quad+ 24n_{\triangle_3}(G_S) 
+ 8n_\square(G) + 24n_\boxtimes(G) \\
&= \sigma + 2\sum_{v\in V(G)} d^2_G(v) - \sum_{v \in V(G)}d_G(v) + 6 \sum_{v_i \in S} d_G(v_i) - 2 \sum_{v_i \in S} n_1(v_i ) \\
&\quad+   8 (n_{\triangle_1}(G_S) + 2 n_{\triangle_2}(G_S) + 3 n_{\triangle_3}(G_S) + n_\square(G) + 3n_\boxtimes(G)). \qedhere
\end{align*}
\end{proof}

\begin{corollary}
\label{closed4G}
When $S=\emptyset$ (i.e., $\sigma=0$), we obtain immediately from Theorem~\ref{closed4GS} that 
\[
w^{cl}_4(G) = 2 \sum_{v \in V(G)} d^2_G(v) - \sum_{v \in V(G)} d_G(v) + 8(n_\square(G) + 3n_\boxtimes(G)),
\]
or equivalently,
\begin{equation}
\label{eq:closed4G}
w^{cl}_4(G) = 2 \left(M_1(G)-m\right) + 8(n_\square(G) + 3n_\boxtimes(G)).
\end{equation}
\end{corollary}

\begin{remark}
\begin{enumerate}[(i)]
\item 
The fourth spectral moment of $G_S,$  $\mathsf{M}_4(G_S)=\sum^n_{i=1} \lambda^4_i(G_S),$ is given by \eqref{eq:closed4GS}.
\item A formula for the fourth spectral moment (for simple graphs) was already reported in \cite[pp 86-87, references therein]{GutmanDas2004}, which reads:
\begin{equation}
\label{eq:closed4GGut}
w^{cl}_4(G)= 2(M_1(G) - m) + 8Q.
\end{equation}
where 
$Q$ is the total number of 4-cycles $C_4$ contained in $G.$ 
Observe that \eqref{eq:closed4G} and \eqref{eq:closed4GGut} coincide because $K_4$ (if any) yields three $C_4$ as illustrated below.  

\begin{figure}[H]
\begin{minipage}[b]{0.23\linewidth}
    \centering
          \begin{tikzpicture}[node distance={15mm}, very thin, main/.style = {draw, circle, fill, inner sep=1.5pt, minimum size=0pt}]  
            \node[main] (1) [label=left:$i$]{}; 
            \node[main] (2) [label=right:$j$][right of=1] {}; 
            \node[main] (3) [below of=1] [label=left:$l$]{}; 
            \node[main] (4) [right of=3] [label=right:$k$]{}; 

            \draw (1) -- (2); 
            \draw (1) -- (3);
            \draw (1) -- (4);
            \draw (2) -- (4);
            \draw (2) -- (3);
            \draw (3) -- (4);
        \end{tikzpicture} 
    \caption*{$K_4$}
\end{minipage}
\begin{minipage}[b]{0.23\linewidth}
    \centering
      \begin{tikzpicture}[node distance={15mm}, very thin, main/.style = {draw, circle, fill, inner sep=1.5pt, minimum size=0pt}]  
            \node[main] (1) [label=left:$i$]{}; 
            \node[main] (2) [right of=1] [label=right:$j$]{}; 
            \node[main] (3) [below of=1] [label=left:$l$]{}; 
            \node[main] (4) [right of=3] [label=right:$k$]{}; 

        \draw (1) -- (2); 
        \draw (1) -- (3);
        \draw (2) -- (4);
        \draw (3) -- (4);
    \end{tikzpicture} 
    \caption*{$C_4$}
\end{minipage}
\begin{minipage}[b]{0.23\linewidth}
    \centering
      \begin{tikzpicture}[node distance={15mm}, very thin, main/.style = {draw, circle, fill, inner sep=1.5pt, minimum size=0pt}] 
                \node[main] (1) [label=left:$i$]{}; 
                \node[main] (2) [right of=1] [label=right:$j$]{}; 
                \node[main] (3) [below of=1] [label=left:$l$]{}; 
                \node[main] (4) [right of=3] [label=right:$k$]{}; 
 
            \draw (1) -- (2);
            \draw (1) -- (4);
            \draw (2) -- (3);
            \draw (3) -- (4); 
        \end{tikzpicture}
    \caption*{$C_4$}
\end{minipage}
\begin{minipage}[b]{0.23\linewidth}
    \centering
      \begin{tikzpicture}[node distance={15mm}, very thin, main/.style = {draw, circle, fill, inner sep=1.5pt, minimum size=0pt}] 
            \node[main] (1) [label=left:$i$]{}; 
                \node[main] (2) [right of=1] [label=right:$j$]{}; 
                \node[main] (3) [below of=1] [label=left:$l$]{}; 
                \node[main] (4) [right of=3] [label=right:$k$]{}; 

                \draw (1) -- (3); 
                \draw (1) -- (4);
                \draw (2) -- (3);
                \draw (2) -- (4); 
        \end{tikzpicture}
    \caption*{$C_4$}
\end{minipage}
\end{figure}

\end{enumerate}


\end{remark}

\begin{example}
Let us illustrate the formula \eqref{eq:closed4GS} with a concrete example. Consider $(K_4)_S$ with $|S|=3,$ with loops at $v_i,v_j,v_l,$ respectively. One can verify that $\sigma=3, M_1(K_4)=36,$ $m=6,$ $\sum_{v_i \in S} d_G(v_i)=9,$ $\sum_{v_i \in S} n_1(v_i)=3$ (formed by edges $v_iv_k,$ $v_jv_k,$ and $v_kv_l$), $n_{\triangle_1}=0,$ $n_{\triangle_2}=3$ (see Figure 2), and $n_{\triangle_3}=1$ (formed by $\triangle(v_i,v_j,v_l)$).  Then, 
\[
w_4^{cl}((K_4)_S) = 3+2(30)+6(9)-2(3)+8(6+3+3) = 207.
\]
Indeed, this coincides with $\sum \lambda^4_i((K_4)_S)=\mathrm{Tr}\:A^4((K_4)_S)= 57 \times 3 + 36=207.$

\begin{figure}[H]
\begin{minipage}[b]{0.21\linewidth}
    \centering
          \begin{tikzpicture}[node distance={15mm}, very thin, main/.style = {draw, circle, fill, inner sep=1.5pt, minimum size=0pt}]  
            \node[main] (1) [label=left:$i$]{}; 
            \node[main] (2) [label=right:$j$][right of=1] {}; 
            \node[main] (3) [below of=1] [label=left:$l$ \vspace{-5em}]{}; 
            \node[main] (4) [right of=3] [label=right:$k$]{}; 

            \draw (1) -- (2); 
            \draw (1) -- (3);
            \draw (1) -- (4);
            \draw (2) -- (4);
            \draw (2) -- (3);
            \draw (3) -- (4);
            \draw (1) to [out=130,in=50,looseness=22] (1);       
            \draw (2) to [out=130,in=50,looseness=22] (2);
            \draw (3) to [out=170,in=90,looseness=20] (3);
        \end{tikzpicture} 
    \caption*{$(K_4)_S$}
    \end{minipage}  
\begin{minipage}[b]{0.21\linewidth}
    \centering
          \begin{tikzpicture}[node distance={15mm}, very thin, main/.style = {draw, circle, fill, inner sep=1.5pt, minimum size=0pt}]  
            \node[main] (1) [label=left:$i$]{}; 
            \node[main] (3) [below of=1] [label=left:$l$ \vspace{-5em}]{}; 
            \node[main] (4) [right of=3] [label=right:$k$]{};

            \draw (1) -- (3);
            \draw (1) -- (4);
            \draw (3) -- (4);
            \draw (1) to [out=130,in=50,looseness=22] (1);    
            \draw (3) to [out=170,in=90,looseness=20] (3);
        \end{tikzpicture} 
    \hspace{-4em} 
    \caption*{$\triangle(v_i,v_j,v_l)$}
    \end{minipage}  
\begin{minipage}[b]{0.21\linewidth}
    \centering
          \begin{tikzpicture}[node distance={15mm}, very thin, main/.style = {draw, circle, fill, inner sep=1.5pt, minimum size=0pt}]  
            \node[main] (1) [label=left:$i$]{}; 
            \node[main] (2) [label=right:$j$][right of=1] {}; 
            \node[main] (4) [right of=3] [label=right:$k$]{}; 

            \draw (1) -- (2); 
            \draw (1) -- (4);
            \draw (2) -- (4);
            \draw (1) to [out=130,in=50,looseness=22] (1);       
            \draw (2) to [out=130,in=50,looseness=22] (2);

        \end{tikzpicture} 
    \hspace{-4em} 
    \caption*{$\triangle(v_i,v_j,v_k)$}
    \end{minipage}    
\begin{minipage}[b]{0.21\linewidth}
    \centering
          \begin{tikzpicture}[node distance={15mm}, very thin, main/.style = {draw, circle, fill, inner sep=1.5pt, minimum size=0pt}]  
            \node[main] (2) [label=right:$j$][right of=1] {}; 
            \node[main] (3) [below of=1] [label=left:$l$]{}; 
            \node[main] (4) [right of=3] [label=right:$k$]{}; 

            \draw (2) -- (4);
            \draw (2) -- (3);
            \draw (3) -- (4);
   
            \draw (2) to [out=130,in=50,looseness=22] (2);
            \draw (3) to [out=130,in=50,looseness=22] (3);
        \end{tikzpicture} 
   \hspace{-4em} 
   \caption*{$\triangle(v_j,v_k,v_l)$}
  \end{minipage}
    \caption{}
  \end{figure}
\end{example}

In the following, we derive the formula for several graph families.

\begin{example}
\label{expKmn}
Let $G=K_{a,b}$ be the complete bipartite graph of parts $(A,B)$ with size $a=|A|,b=|B| \geq 1.$ Since $K_{a,b}$ contains only even cycles, the only contribution in closed $4$-walks is by $n_{\square}(K_{a,b}).$ It is known that the number of $2k$-cycles in $K_{a,b}$ is given by 
\[
\begin{pmatrix}
b \\ k 
\end{pmatrix}
\begin{pmatrix}
a \\ k 
\end{pmatrix}
\frac{(k-1)!k!}{2}.
\]
For a $4$-cycle, i.e. when $k=2,$ we have 
\[
n_{\square}(K_{a,b})= \frac{1}{4} ab(a-1)(b-1).
\]
Since $\sum_{v \in V(G)} d_G^2(v) = ab^2 + ba^2$ and $\sum_{v \in V(G)} d_G(v)= 2ab,$ by a direct computation, we obtain 
\[
w^{cl}_4(K_{a,b}) = 2a^2b^2.
\]
Let $S \subseteq V(G)$ with $|S|=\sigma=\sigma_A + \sigma_B.$ Then, 
\begin{align*}
\sum_{v_i \in S} d_G (v_i) &= b\sigma_A + a\sigma_B, \\
\sum_{v_i \in S} n_1(v_i) &= \sigma_A(b-\sigma_B) + \sigma_B(a-\sigma_A)= b\sigma_A + a\sigma_B -2\sigma_A\sigma_B. 
\end{align*}
Combining all, we obtain 
\begin{equation}
w_4^{cl}((K_{a,b})_S) = \sigma_A(4b+1) + \sigma_B(4a+1) + 4\sigma_A\sigma_B + 2a^2b^2,
\end{equation}
for which one observes that $w_4^{cl}((K_{a,b})_S)= w_4^{cl}(K_{a,b})$ when $S=\emptyset.$ Such $w_4^{cl}((K_{a,b})_S)$ is independent of the location of loops in any parts of vertices.
\end{example}


\begin{example}
\label{expPn}
Let $G=P_n$ be the path of $n$ vertices. It is immediate to obtain $M_1(G)=4n-6$ and thus $2(M_1(G)-m)=2(3n-5).$ Consider $(P_n)_S$ with $0\leq \sigma \leq n.$ Then,
\[
\sum_{v_i \in S} d_G(v_i)=2\sigma_{ne} + \sigma_e,
\] 
where $\sigma_e$ (resp. $\sigma_{ne}$) is the number of endpoint vertices (resp. non-endpoint) with a loop attached.

Suppose all loops are non-adjacent to each other, then $\sum_{v_i\in S} n_1(v_i) = 2\sigma_{ne} + \sigma_e.$ It follows that
\begin{align*}
w_4^{cl}((P_n)_S) 
&= 2(3n-5)+\sigma + 4(2\sigma_{ne}+\sigma_e).
\end{align*}

Suppose $\sigma_e=0.$ Without loss of generality, let $n \geq 4.$ If there are adjacent loops in the sense whose neighborhood contains at least one loop, then 
$
\sum_{v_i \in S} n_1(v_i) = 2 (n_{P'_{k}} + n_{na}),
$ where $n_{P'_k}$ be the total number of paths of adjacent $k$-loops for $k=2,3,...$, and $\sigma_{na}$ the number of non-adjacent loops. Thus, 
\[
w_4^{cl}((P_n)_S) =2(3n-5) + 13\sigma - 4(n_{P'_{k}} + \sigma_{na}).
\] 

To illustrate this, consider the following self-loop paths. 
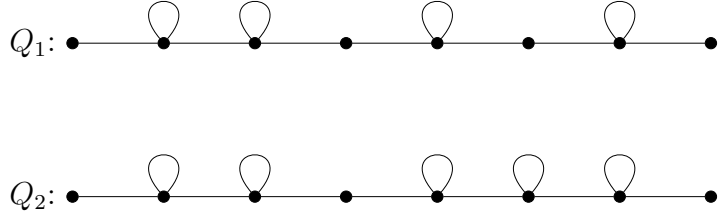
\begin{figure}[H]
 \begin{center}
          \begin{tikzpicture}[node distance={12mm}, very thin, main/.style = {draw, circle, fill, inner sep=1.5pt, minimum size=0pt}]  
            \node[main] (1) [label=left:$Q_1$:]{}; 
            \node[main] (2) [right of=1]{}; 
            \node[main] (3) [right of=2]{}; 
            \node[main] (4) [right of=3]{}; 
            \node[main] (5) [right of=4]{};
            \node[main] (6) [right of=5]{}; 
            \node[main] (7) [right of=6]{}; 
            \node[main] (8) [right of=7]{}; 

            \draw (1) -- (2); 
            \draw (2) -- (3);
            \draw (3) -- (4);
            \draw (4) -- (5);
            \draw (5) -- (6);
            \draw (6) -- (7);
            \draw (7) -- (8);
            \draw (2) to [out=130,in=50,looseness=22] (2);       
            \draw (3) to [out=130,in=50,looseness=22] (3);
            \draw (5) to [out=130,in=50,looseness=22] (5);
            \draw (7) to [out=130,in=50,looseness=22] (7);
        \end{tikzpicture} 
    \end{center}
\end{figure}
\begin{figure}[H]
 \begin{center}
          \begin{tikzpicture}[node distance={12mm}, very thin, main/.style = {draw, circle, fill, inner sep=1.5pt, minimum size=0pt}]  
            \node[main] (1) [label=left:$Q_2$:]{}; 
            \node[main] (2) [right of=1]{}; 
            \node[main] (3) [right of=2]{}; 
            \node[main] (4) [right of=3]{}; 
            \node[main] (5) [right of=4]{};
            \node[main] (6) [right of=5]{}; 
            \node[main] (7) [right of=6]{}; 
            \node[main] (8) [right of=7]{}; 

            \draw (1) -- (2); 
            \draw (2) -- (3);
            \draw (3) -- (4);
            \draw (4) -- (5);
            \draw (5) -- (6);
            \draw (6) -- (7);
            \draw (7) -- (8);
            \draw (2) to [out=130,in=50,looseness=22] (2);       
            \draw (3) to [out=130,in=50,looseness=22] (3);
            \draw (5) to [out=130,in=50,looseness=22] (5);
            \draw (6) to [out=130,in=50,looseness=22] (6);
            \draw (7) to [out=130,in=50,looseness=22] (7);
        \end{tikzpicture} 
    \end{center}
\end{figure}
For path $Q_1,$ $n_{P'_k}=1$ and $\sigma_{na}=2,$ giving $w_4^{cl} = 2(3(8)-5)+13(4)- 4(1+2)=78.$ On the other hand, for path $Q_2,$ we have $n_{P'_k}=2$ (of length 2 and 3, respectively) and $\sigma_{na}=0.$ Thus, $w_4^{cl} = 2(3(8)-5)+13(5)- 4(2+0)=95.$

The case for $\sigma_e\neq 0$ can be deduced in a similar method, and is left as exercise to interested reader.
\end{example}

\begin{example}
Recall from \cite{GutmanDas2004} that for a star $S_n=K_{1,n-1},$ a path $P_n,$ and a tree $T_n$ (different from the star or path) of order $n \geq 5,$ the inequality holds:
\[
M_1(P_n) < M_1(T_n) < M_1(S_n).
\]
It is straightforward to observe that these three graphs have $m=n-1,$ and $n_{\triangle_1}=n_{\triangle_2}=n_{\triangle_3}=n_\square = n_\boxtimes =0.$  By Example~\ref{expKmn}, for parts $(A,B)$ with $|A|=1$ and $|B|=n-1,$ we have
\begin{align*}
w_4^{cl}((S_n)_S) 
&=\begin{cases}
2(n-1)^2 +5\sigma_B, & \sigma_A=0, \\
2(n-1)^2 +9\sigma_B + 4n-3, & \sigma_A=1.
\end{cases}
\end{align*}
When $\sigma=0,$ by the formula in Example~\ref{expPn}, we obtain the bound for $n\geq 5:$
\[
2(3n-5) < w_4^{cl}(T_n) < 2(n-1)^2.
\]
\end{example}

\begin{example}
Let $G_S=(C_n)_S$ be the cycle graph of order $n \geq 3$ with $\sigma$ loops. It is clear that $2(M_1-m)=6n$ and $6\sum_{v\in S}d_G(v)=12\sigma.$

\underline{Case 1}: Let $n\geq 5.$ Then, $n_{\triangle_1}=n_{\triangle_2}=n_{\triangle_3}=n_{\square} =n_{\boxtimes}=0.$  Let $n_{P'_k}$ be the number of paths of adjacent $k$-loops, and $\sigma_{na}$ be the number of non-adjacent loops, then 
\[
w_4^{cl}((C_n)_S) = 6n + 13\sigma - 4\left( n_{P'_k} + \sigma_{na} \right), \quad n\geq 5.
\]

\underline{Case 2}: Let $n=4.$ Then, $n_\square=1,$ and we get
\[
w_4^{cl}((C_4)_S) = 32 + 13\sigma - 4\left(n_{P'_k} + \sigma_{na} \right).
\]

\underline{Case 3}: Let $n=3.$ It is straightforward to obtain $w_4^{cl}((C_3)_S)=35 \:(\text{with }n_{P'_k}=0, \sigma_{na}=1, n_{\triangle_1}=1),$ $56 \:(\text{with }n_{P'_k}=1, \sigma_{na}=0, n_{\triangle_2}=1),$ and $81 \:(\text{with } n_{P'_k}=0, \sigma_{na}=0, n_{\triangle_3}=1)$ for $\sigma=1,2,$ and $3$ respectively.
\end{example}

\begin{example}
Let $G=K_n, n\geq 4.$ Since $K_n$ is regular, we have
\[
2(M_1(G)-m)=2\left(n(n-1)^2 - \frac{n(n-1)}{2}\right) = 2n(n-1)\left(n-\frac{3}{2}\right),\quad n\geq 4.
\]
It suffices to determine the number of 4-cycles $C_4$ in $K_n.$ Consider any $C_4$ with vertices $v_i,v_j,v_k,v_l.$ Without loss of generality, consider a closed 4 walk $ijkli.$ There are $(4-1)!/2$ ways to permute $v_j,v_k,$ and $v_l.$ 
Thus, we have 
\[
\frac{(4-1)!}{2}\begin{pmatrix}
    n \\ 4
\end{pmatrix} = \frac{n!}{8(n-4)!} 
\] many distinct $C_4$'s in $K_n.$ In total, we have
\[
w_4^{cl}(K_n) = 2n(n-1)\left(n-\frac{3}{2}\right)+\frac{n!}{(n-4)!}, \quad n \geq 4.
\] 
\end{example}


\section{Energy of self-loop graphs and twisted moments}
\label{sect3}

As it is known that the $k$-th spectral moment $\mathsf{M}_k(G_S)=\mathsf{M}_k(A(G_S))$ associated to a self-loop graph $G_S$ coincides with the number of closed $k$-walks on $G_S,$ i.e., $\mathsf{M}_k(G_S)=w^{cl}_k(G_S).$ The main goal of this section is to investigate an extension to some moment-like quantities \textit{twisted} by $\mathsf{M}_k(G_S)$. For brevity, we shall call these quantities \textit{twisted moments}. 

\begin{definition}\label{def3.1}
Let $B$ be an $n \times n$ real symmetric matrix and $\lambda_1(B) \geq \lambda_2(B) \geq \cdots \geq \lambda_n(B)$ be its eigenvalues. For $q \in \R, k \in \N,$ define the \textit{$\mathsf{M}_k(B)$-twisted moment} of $B,$ denoted by $\cM^k_q(B),$ as 
\[
\cM^k_q(B) = \sum^n_{i=1} \left|\lambda_i(B) - \frac{\mathsf{M}_k(B)}{n} \right|^q \:\: \in \R.  
\]
\end{definition}

\begin{remark}
Henceforth, we shall consider $B$ in Definition~\ref{def3.1} as the adjacency matrix $A(G_S)$ of a self-loop graph. When $k=1,$ the twisting is 
\[
\mathsf{M}_1(G_S)=\tr(A(G_S))=\sigma.
\]
The $\mathsf{M}_1(G_S)$-twisted moment of $G_S$ is
\begin{equation}
\label{eq:MqGS}
\cM_q(G_S) := \cM^1_q(A(G_S))= \sum^n_{i=1} \left|\lambda_i(G_S) - \frac{\sigma}{n} \right|^q.    
\end{equation}
\end{remark}
In the following, we derive several formulae of  $\cM_q(G_S)$ for $q=0,1,2,3,4,$ which may be of independent interest. The first three are straightforward: 
\begin{align}
    \cM_0(G_S)&=n, \label{eq:M0eq}\\
    \cM_1(G_S)&=\cE(G_S), \label{eq:M1eq}\\
    \cM_2(G_S)&= 2m + \sigma - \frac{\sigma^2}{n} 
    = w_2(G_S) - \frac{\sigma^2}{n}.
    \label{eq:M2eq}
\end{align}

\begin{proposition}
\label{M34formula}
Let $G_S$ be a self-loop graph of order $n$ and $|S|=\sigma.$ Let $j \in \N$ be such that $\lambda_1(G_S) \geq \lambda_2(G_S) \geq \cdots \geq \lambda_j(G_S) \geq \frac{\sigma}{n}.$ Then,
    \begin{align}
    \cM_3(G_S)&= 2\sum^j_{i=1} \lambda_i^3 - \frac{6\sigma}{n} \sum^j_{i=1} \lambda_i^2 + \frac{4\sigma^2}{n^2} \sum^j_{i=1} \lambda_i  - w_3^{cl}(G_S)  + \frac{3\sigma}{n} w_2^{cl}(G_S) - \frac{2\sigma^3}{n^2} + \frac{\sigma^2}{n^2} \cE(G_S), \label{eq:M3eq}\\
    \cM_4(G_S)&= w_4^{cl}(G_S) - \frac{4\sigma}{n} w_3^{cl}(G_S) + \frac{6 \sigma^2}{n^2} w_2^{cl}(G_S) - \frac{3\sigma^4}{n^3}, \label{eq:M4eq}
    \end{align}    
where $w_2^{cl}(G_S), w^{cl}_3(G_S),$ and $w_4^{cl}(G_S)$ denote the number of closed 2-,3-, and 4-walks on $G_S,$ respectively, as obtained in the previous section.
\end{proposition}

\vspace{-1em}

\begin{proof}
We first derive the latter. By Binomial Theorem,
\begin{align*}
\cM_4(G_S)
&= \sum^n_{i=1} \left|\lambda_i(G_S) - \frac{\sigma}{n} \right|^4 \\
&= \sum^n_{i=1} \left( \lambda_i^4(G_S) - \frac{4\sigma}{n}\lambda_i^3(G_S) +  \frac{6\sigma^2}{n^2}\lambda_i^2(G_S) - \frac{4\sigma^3}{n^3}\lambda_i(G_S) + \frac{\sigma^4}{n^4} \right) \\
&= \sum^n_{i=1} \lambda_i^4(G_S) - \frac{4\sigma}{n}\sum^n_{i=1}\lambda_i^3(G_S) +  \frac{6\sigma^2}{n^2}\sum^n_{i=1}\lambda_i^2(G_S) - \frac{4\sigma^3}{n^3}\sum^n_{i=1}\lambda_i(G_S) + \frac{\sigma^4}{n^4}\sum^n_{i=1} 1 \\
&= w_4^{cl}(G_S) - \frac{4\sigma}{n}w_3^{cl}(G_S) +  \frac{6\sigma^2}{n^2}w_2^{cl}(G_S) - \frac{3\sigma^4}{n^3}. 
\end{align*}
We shall now derive the formula for $\cM_3(G_S)$. For simplicity, we write $\lambda_i=\lambda_i(G_S).$  Let $j \in \N$ be such that $\lambda_1 \geq \lambda_2 \geq \cdots \geq \lambda_j \geq \frac{\sigma}{n}$ and $\frac{\sigma}{n}>\lambda_{j+1} \geq \cdots \lambda_n.$ Then, 
\[
\lambda^2_i\left| \lambda_i - \frac{\sigma}{n}\right|= 
\begin{cases}
  \lambda_i^3 - \lambda_i^2\dfrac{\sigma}{n}, & i=1,\ldots,j \\
  \lambda_i^2\dfrac{\sigma}{n} - \lambda_i^3, & i=j+1,\ldots, n.
\end{cases}
\] 
Since 
\[
-\sum^n_{i=j+1} \lambda_i^3 = \sum^j_{i=1}\lambda_i^3 - w_3(G_S) \;\;\text{ and }\;\;
-\dfrac{\sigma}{n} \left(\sum^j_{i=1} \lambda_i^2 - \sum^n_{i=j+1}\lambda_i^2\right) = 
- \dfrac{2\sigma}{n}\sum^j_{i=1} \lambda_i^2 + \frac{\sigma}{n}w_2^{cl}(G_S),
\]
we obtain
\begin{align*}
\sum^n_{i=1}\lambda^2_i\left| \lambda_i - \frac{\sigma}{n}\right|
&=  \left( \sum^j_{i=1} \lambda_i^3 - \dfrac{\sigma}{n}\sum^j_{i=1}\lambda_i^2 \right) + \left( \dfrac{\sigma}{n}\sum^n_{i=j+1}\lambda_i^2 - \sum^n_{i=j+1} \lambda_i^3  \right) \\
&= 2\sum^j_{i=1} \lambda_i^3 - w_3^{cl}(G_S)  + \frac{\sigma}{n}w_2^{cl}(G_S) - \frac{2\sigma}{n} \sum^j_{i=1} \lambda_i^2 . 
\end{align*}
Using similar method, we deduce that 
\[
\sum^n_{i=1}\lambda_i\left| \lambda_i - \frac{\sigma}{n}\right|
= 2\sum^j_{i=1} \lambda_i^2 - \frac{2\sigma}{n} \sum^j_{i=1} \lambda_i  - w_2^{cl}(G_S) + \frac{\sigma^2}{n}.
\]
Since $\cE(G_S)= \sum^n_{i=1}|\lambda_i - \frac{\sigma}{n}|,$ we obtain
\begin{align*}
\cM_3(G_S)
&= \sum^n_{i=1}\left(\lambda^2_i -\frac{2\sigma}{n} \lambda_i +\frac{\sigma^2}{n^2}\right) \left| \lambda_i - \frac{\sigma}{n}\right| \\
&=\sum^n_{i=1}\lambda^2_i\left| \lambda_i - \frac{\sigma}{n}\right| -\frac{2\sigma}{n} \sum^n_{i=1}\lambda_i\left| \lambda_i - \frac{\sigma}{n}\right|  + \frac{\sigma^2}{n^2} \sum^n_{i=1}\left| \lambda_i - \frac{\sigma}{n}\right| \\
&= 2\sum^j_{i=1} \lambda_i^3 - \frac{6\sigma}{n} \sum^j_{i=1} \lambda_i^2 + \frac{4\sigma^2}{n^2} \sum^j_{i=1} \lambda_i  - w_3^{cl}(G_S) + \frac{3\sigma}{n} w_2^{cl}(G_S)  - \frac{2\sigma^3}{n^2} + \frac{\sigma^2}{n^2} \cE(G_S) .
\end{align*}
\end{proof}

\vspace{-0.2em}
\begin{theorem}
\label{mainInq}
Let $G_S$ be a connected self-loop graph with $|S|=\sigma \geq 1$.  Let $p, q \in \R$ with $p \leq q.$ Then, 
\begin{equation}
\cM_q(G_S)^2 \leq \cM_{2q -2p}(G_S) \cM_{2p}(G_S).
\end{equation}
\end{theorem}

\begin{proof} 
Suppose $\lambda_i \neq \frac{\sigma}{n}$ for all $i=1,...,n.$ Then, $|\lambda_i- \frac{\sigma}{n}|\neq 0$ for all $i=1,...,n.$ 

\noindent By the Cauchy-Schwarz inequality,
\begin{align*}
\cM_q(G_S)
&= \sum^n_{i=1} \left| \lambda_i-\frac{\sigma}{n}\right|^{q-p} \left| \lambda_i-\frac{\sigma}{n}\right|^{p}
&\leq \left[ \left(\sum^n_{i=1} \left| \lambda_i-\frac{\sigma}{n}\right|^{2q-2p}\right) \left( \sum^n_{i=1}\left| \lambda_i-\frac{\sigma}{n}\right|^{2p}\right) \right]^{\frac12}.   
\end{align*}
If there exist $j_k\in \{1,...,n\}$ such that $\lambda_{j_k} =\frac{\sigma}{n},$ then $|\lambda_{j_k} -\frac{\sigma}{n}|=0$ for all such $j_k$'s.  Thus,
\begin{align*}
\cM_q(G_S)
&= \sum^n_{i=1, i \neq j_k} \left| \lambda_i-\frac{\sigma}{n}\right|^{q-p} \left| \lambda_i-\frac{\sigma}{n}\right|^{p} \\
&\leq \left[ \left(\sum^n_{i=1, i \neq j_k} \left| \lambda_i-\frac{\sigma}{n}\right|^{2q-2p}\right) \left( \sum^n_{i=1, i \neq j_k}\left| \lambda_i-\frac{\sigma}{n}\right|^{2p}\right) \right]^{\frac12} \\
&= \left[ \left(\sum^n_{i=1} \left| \lambda_i-\frac{\sigma}{n}\right|^{2q-2p}\right) \left( \sum^n_{i=1}\left| \lambda_i-\frac{\sigma}{n}\right|^{2p}\right) \right]^{\frac12} = \left(\cM_{2q-2p}(G_S)\cM_{2p}(G_S)\right)^{\frac{1}{2}},
\end{align*}
where the first equality in the third line follows by adding $j_k$ many $0^{2r}=0$ (with $r=q-p$) in each summation, and this expansion does not affect the product.
\end{proof}

Observe that by Theorem~\ref{mainInq}  with $q=1, p=1,$ yields the McClelland-type bound for $G_S$ \cite{gutman2021energy}: 
\[
\cE(G_S) \leq \sqrt{n\left(2m+\sigma - \frac{\sigma^2}{n} \right)}. 
\]
Moreover, it also follows immediately from Theorem~\ref{mainInq} that 
\begin{equation}
\label{eq:IneqMk1}
\frac{\cM_q(G_S)^2}{n} \leq \cM_{2q}(G_S),    
\end{equation}    
\begin{equation}
\label{eq:IneqMk2}
\frac{\cM_q(G_S)^4}{n^3} \leq \cM_{4q}(G_S).    
\end{equation}    
Next, we establish a positivity result about the twisted moment $\cM_i.$
\begin{theorem}
\label{M2}
Let $G_S$ be a connected self-loop graph of order $n\geq 2,$ size $m \geq 1,$ and $|S|=\sigma$ where $0 \leq \sigma \leq n.$  Then, $\cM_i(G_S)>0$ for all $i\in \N \cup \{0\}.$
\end{theorem}
\begin{proof}
For simplicity, we write $\cM_i=\cM_i(G_S).$
The first two $\cM_0=n$ and $\cM_1=\cE(G_S)$ are clear. Note that $\cM_2=0$ if and only if $m=0$ and $\sigma=0,$ i.e. $G_S$ is an edgeless and loopless graph  $\overline{K_n}$. On the other hand, for the case $2m+\sigma < \frac{\sigma^2}{n}$ to hold, $\sigma$ needs to be maximized. When $\sigma=n,$ we get $2m<0,$ a contradiction. Thus, $\cM_2>0$ for any connected self-loop graphs.  Now, by taking  $q=2, p=3/2$ in Theorem~\ref{mainInq}, we have 
\[
\cM_3 \geq \frac{\cM_2^2}{\cM_1} \geq \sqrt{\frac{\cM_2^3}{n}}>0.
\]
By taking $q=3, p=2$ in Theorem~\ref{mainInq}, it is then clear that $\cM_4\geq \cM^2_3/\cM_2 >0.$ For $i \in \N, i\geq 5,$  by induction we deduce that 
\[
\cM_i \geq \frac{\cM^2_{i-1}}{\cM_{i-2}} >0. 
\]
For the loopless case, it suffices to
consider $\cM_3(G)= (\sum^j_{i=1} 2\lambda_i^3(G)) - w_3^{cl}(G),$ where $j \in \N$ such that $\lambda_j \geq 0.$ Even when $w_3^{cl}(G)=0,$ we have $\cM_3(G)\neq 0$ by the connectivity of $G.$  
\end{proof}

Actually, the previous proof leads to a nice ratio property. Consider any connected $G_S$ with $|S|=\sigma \geq 0.$ 
Let $k\in \N.$ Write $\cM_{q}=\cM_q(G_S).$ When $q=2k-1, p=k,$ we have $\dfrac{\cM_{2k-1}}{\cM_{2k-2}} \leq \dfrac{\cM_{2k}}{\cM_{2k-1}}.$ When $q= 2k, p=k+\frac12,$ then $\dfrac{\cM_{2k}}{\cM_{2k-1}} \leq \dfrac{\cM_{2k+1}}{\cM_{2k}}.$ By Theorem~\ref{M2}, all such fractions are well-defined. Thus, combining these two cases, we have the following result.

\begin{theorem}
\label{mainInq2}
Let $G_S$ be a connected self-loop graph with $|S|=\sigma,$ where $0 \leq \sigma \leq n.$ 
Then,
\begin{equation}
\label{eq:ratio}
\frac{\cM_{1}(G_S)}{\cM_{0}(G_S)} \leq \frac{\cM_{2}(G_S)}{\cM_{1}(G_S)} \leq 
\frac{\cM_{3}(G_S)}{\cM_{2}(G_S)} \leq \frac{\cM_{4}(G_S)}{\cM_{3}(G_S)} \leq \cdots \leq \frac{\cM_{n}(G_S)}{\cM_{n-1}(G_S)} \leq \cdots.  
\end{equation}
\end{theorem}
\vspace{1em}

\begin{corollary}
\label{lb1}
Let $G_S$ be a connected self-loop graph with $|S|=\sigma,$ where $0 \leq \sigma \leq n.$ Then,
\begin{equation}
\label{eq:lb1}
\cE(G_S) \geq \sqrt{\frac{\cM_2^3}{\cM_4}}.
\end{equation}
In particular, the equality holds if $G_S \cong (K_{a,b})_S$ when $\sigma=0$ and $\sigma=n.$  
\end{corollary}

\begin{proof}
By \eqref{eq:ratio}, from $\dfrac{\cM_2}{\cM_1} \leq \dfrac{\cM_3}{\cM_2}$ we have 
$
\cE(G_S)=\cM_1 \geq \dfrac{\cM_2^2}{\cM_3} > 0.
$
Similarly, from $\dfrac{\cM_2}{\cM_1} \leq \dfrac{\cM_4}{\cM_3}$ we have 
$
\cE(G_S)=\cM_1 \geq \dfrac{\cM_2\cM_3}{\cM_4} > 0.
$ 
It follows that 
\[
\cE(G_S)^2 \geq \dfrac{\cM_2^2}{\cM_3} \cdot \dfrac{\cM_2\cM_3}{\cM_4} = \dfrac{\cM_2^3}{\cM_4}. 
\]
For equality, we shall only discuss the non-trivial case $G_S \cong \widehat{K_{a,b}}.$  Observe that from \eqref{eq:M2eq} and \eqref{eq:M4eq}, we have $\cM_2(\widehat{K_{a,b}})=2ab$ and $\cM_4(\widehat{K_{a,b}})= 2(ab)^2,$ respectively. On the other hand, by \cite[Theorem 2.4, Case 5]{akbari2023selfloop}, we have $\cE(\widehat{K_{a,b}})=2\sqrt{ab}.$ Thus we obtain $\cE(\widehat{K_{a,b}})^2= \cM^3_2(\widehat{K_{a,b}})/\cM_4(\widehat{K_{a,b}}).$ 
\end{proof}

When $\sigma=0,$ it follows from \eqref{eq:M2eq} that $\cM_2(G)=w^{cl}_2(G)$ and \eqref{eq:M4eq} that $\cM_4(G)= w^{cl}_4(G).$ Thus, Corollary~\ref{lb1} can be considered as an extension of the classical lower bound 
\[
\cE(G) \geq 2\sqrt{2}m\sqrt{\frac{m}{w^{cl}_4(G)}},
\] (cf. \cite[\S 4, eq (11)]{MajKloGut2009} and references therein) to self-loop graphs in terms of twisted moments.

The formulae of $\cM_3$ and $\cM_4$ from Proposition~\ref{M34formula} are exact but somewhat tedious. Here, we give a lower bound in terms of order $n$ and size $m$ only. Recall that:

\begin{corollary}\cite[Corollary 3.6]{akbari2024line} \label{Cor3.8}
  Let $G$ be a connected graph of order $n \geq 2$ and size $m,$ and let $S \subseteq V(G)$ and $|S|=\sigma, 0 \leq \sigma \leq n.$ Then,
  \[
    \cE(G_S) \geq \frac{4m}{n}.
  \]
\end{corollary}

\begin{corollary}
\label{lb2}
Let $G_S$ be a connected self-loop graph of order $n$ and size $m.$ Then,
\begin{equation}
\label{eq:lb2}
\cM_3(G_S) \geq \frac{64m^3}{n^5}, \qquad
\cM_4(G_S) \geq \frac{256m^4}{n^7}.
\end{equation}
\end{corollary}

\begin{proof}
By Theorem~\ref{mainInq2}, we obtain 
\begin{align*}
\cM_3(G_S) \geq \frac{\cM_2^2(G_S)}{\cE(G_S)} \geq \frac{\cE(G_S)^3}{n^2}, \qquad 
\cM_4(G_S) \geq \frac{\cM_2^2(G_S)}{n} \geq \frac{\cE(G_S)^4}{n^3}.
\end{align*}
The claim follows immediately by Corollary~\ref{Cor3.8}. 
\end{proof}

Next, we prove a generalization of \cite[Theorem 1a]{Zhou2007} to self-loop graphs. 

\begin{theorem}
Let $G_S$ be a self-loop graph of order $n \geq 2$ and size $m\geq 1.$ Let $r,s,t$ be nonnegative real numbers such that $4r=s+t+2.$ Then,
\[
\cE(G_S) \geq \frac{\cM^2_r(G_S)}{\sqrt{\cM_s(G_S) \cM_t(G_S)}}.
\]
\end{theorem}

\begin{proof}
Let $q=r, p=\frac{1}{2},$ and $k=2r-1 = 2q-2p.$ By Theorem~\ref{mainInq}, we obtain 
\begin{equation} \label{eq:3.9a}
\cM_r^2(G_S) \leq \cM_1(G_S) \cM_{k}(G_S).
\end{equation}
By assumption, we deduce $k=\frac{1}{2}(s+t).$ Apply Theorem 3.4 again, we have 
\begin{equation}\label{eq:3.9b}
\cM^2_{k}(G_S) \leq \cM_s(G_S)\cM_t(G_S).
\end{equation} 
Squaring \eqref{eq:3.9a} and combine with \eqref{eq:3.9b},
\[
\cM_r^{4} \leq \cM_1^2(G_S) \cM_k^2(G_S) \leq \cM_1^2(G_S)\cM_s(G_S)\cM_t(G_S). 
\] Since $\cM_1(G_S)=\cE(G_S),$ we get 
\[
\cE(G_S) \geq \frac{\cM^2_r(G_S)}{\sqrt{\cM_s(G_S) \cM_t(G_S)}}.\qedhere
\]
\end{proof}

\vspace{0.5cm}
\textbf{Acknowledgement.}
The author thanks Irena M. Jovanov\'ic for encouragement and discussion on the topic, especially pointing out the relevance of Theorem 2.1, and corrections in Example 2.17 and Proposition 3.3. The author is also grateful to the reviewers for their constructive suggestions and helpful advice that greatly improves the presentation of this paper.

\vspace{0.5cm}
\textbf{Conflicts of interest.} The author declares no conflict of interest.

\bibliography{bibliography}{}

\providecommand{\bysame}{\leavevmode\hbox to3em{\hrulefill}\thinspace}
\providecommand{\MR}{\relax\ifhmode\unskip\space\fi MR }
\providecommand{\MRhref}[2]{%
  \href{http://www.ams.org/mathscinet-getitem?mr=#1}{#2}
}
\providecommand{\href}[2]{#2}
\begin{thebibliography}{10}

\bibitem{Aigner2018}
M.~Aigner and G.~M. Ziegler, \emph{In praise of inequalities}, pp.~143--150,
  Springer Berlin Heidelberg, 2018.

\bibitem{akbari2023selfloop}
S.~Akbari, H.~Al~Menderj, M.~H. Ang, J.~Lim, and Z.~C. Ng, \emph{Some results
  on spectrum and energy of graphs with loops}, Bull. Malays. Math. Sci. Soc.
  \textbf{46} (2023), no.~3, Paper No. 94, 18. \MR{4567384}

\bibitem{akbari2024line}
S.~Akbari, I.~M. Jovanovi\'c, and J.~Lim, \emph{Line graphs and
  {N}ordhaus-{G}addum-type bounds for self-loop graphs}, Bull. Malays. Math.
  Sci. Soc. \textbf{47} (2024), no.~4, Paper No. 117, 22. \MR{4751718}

\bibitem{biggs1993algebraic}
N.~Biggs, \emph{Algebraic {G}raph {T}heory}, no.~67, Cambridge University
  Press, 1993.

\bibitem{cvetkovic1995spectra}
D.~M. Cvetkovi\'{c}, M.~Doob, and H.~Sachs, \emph{Spectra of graphs}, third
  ed., Johann Ambrosius Barth, Heidelberg, 1995, Theory and applications.
  \MR{1324340}

\bibitem{GutmanDas2004}
I.~Gutman and K.~C. Das, \emph{The first {Z}agreb index 30 years after}, MATCH
  Commun. Math. Comput. Chem. (2004), no.~50, 83--92. \MR{2037426}

\bibitem{gutman2021energy}
I.~Gutman, I.~Red{\v{z}}epovi{\'c}, B.~Furtula, and A.~Sahal, \emph{Energy of
  {G}raphs with {S}elf-{L}oops}, MATCH Commun. Math. Comput. Chem. \textbf{87}
  (2021), 645--652.

\bibitem{Harary1969}
F.~Harary, \emph{Graph theory}, Addison-Wesley Publishing Co., Reading,
  Mass.-Menlo Park, Calif.-London, 1969. \MR{256911}

\bibitem{HornJohnson2013}
R.~A. Horn and C.~R. Johnson, \emph{Matrix analysis}, second ed., Cambridge
  University Press, Cambridge, 2013. \MR{2978290}

\bibitem{jovanovic2023}
I.~Jovanovi{\'c}, E.~Zogi{\'c}, and E.~Glogi{\'c}, \emph{On the conjecture
  related to the energy of graphs with self-loops}, MATCH Commun. Math. Comput.
  Chem. \textbf{89} (2023), 479--488.

\bibitem{SomborQuanti2023}
J.~Lim, Z.~K. Chew, M.~Lim, and K.~J. Thoo, \emph{Quantization of {S}ombor
  {E}nergy for {C}omplete {G}raphs with {S}elf-{L}oops of {L}arge {S}ize},
  Iranian Journal of Mathematical Chemistry \textbf{14} (2023), no.~4,
  225--241.

\bibitem{MajKloGut2009}
S.~Majstorovi\'{c}, A.~Klobu\v{c}ar, and I.~Gutman, \emph{Selected topics from
  the theory of graph energy: hypoenergetic graphs}, Zb. Rad. (Beogr.)
  \textbf{13(21)} (2009), 65--105. \MR{2543254}

\bibitem{rakshith2024}
B.~R. Rakshith, K.~C. Das, B.~J. Manjunatha, and Y.~Shang, \emph{Relations
  between ordinary energy and energy of a self-loop graph}, Heliyon \textbf{10}
  (2024), no.~6, e27756.

\bibitem{ShettyBhat2023}
S.~S. Shetty and A.~K. Bhat, \emph{On the first zagreb index of graphs with
  self-loops}, AKCE International Journal of Graphs and Combinatorics
  \textbf{20} (2023), no.~3, 326--331.

\bibitem{Zhou2007}
B.~Zhou, I.~Gutman, J.~A. de~la Pe\~{n}a, J.~Rada, and L.~Mendoza, \emph{On
  spectral moments and energy of graphs}, MATCH Commun. Math. Comput. Chem.
  \textbf{57} (2007), no.~1, 183--191. \MR{2293903}

\end{thebibliography}
\bibliographystyle{amsplain}


\end{document}